\documentclass[12pt,leqno]{amsart}
\usepackage{amsmath,amsfonts,amssymb,amsthm}
\usepackage{graphicx}
\graphicspath{ }
\usepackage{wrapfig}
\usepackage{accents}
\usepackage{caption}
\usepackage{subcaption}
\usepackage{calligra}

\numberwithin{figure}{section}

\theoremstyle{plain}
\newtheorem{thm}{Theorem}[section]
\newtheorem{lem}[thm]{Lemma}

\newtheorem{cor}{Corollary}[thm]
\theoremstyle{definition}

\theoremstyle{remark}

\usepackage{mathtools}
\title[Isometry theorem of gradient Shrinking Ricci solitons]{Isometry theorem of gradient Shrinking Ricci solitons}
\author[A. A. Shaikh, C. K. Mondal ]{Absos Ali Shaikh$^1$ and Chandan Kumar Mondal$^2$ }
\address{\noindent\newline $^{12}$Department of Mathematics,\newline University of
Burdwan, Golapbag,\newline Burdwan-713104,\newline West Bengal, India}
\email{$^1$aask2003@yahoo.co.in, aashaikh@math.buruniv.ac.in}
\email{$^2$chan.alge@gmail.com}

\begin{document}
\begin{abstract}
In this paper, we have proved that if a complete conformally flat gradient shrinking Ricci soliton has linear volume growth or the scalar curvature is finitely integrable and also the reciprocal of the potential function is subharmonic, then the manifold is isometric to the Euclidean sphere. As a consequence, we have showed that a four dimensional  gradient shrinking Ricci soliton satisfying some conditions is isometric to $\mathbb{S}^4$ or $\mathbb{RP}^4$ or $\mathbb{CP}^2$. We have also deduced a condition for the shrinking Ricci soliton to be compact with quadratic volume growth.
\end{abstract}
\noindent\footnotetext{$\mathbf{2020}$\hspace{5pt}Mathematics\; Subject\; Classification: 53C20; 53C21 }
\maketitle
\section{Introduction and results}
 A complete Riemannian manifold $(M,g)$, of dimension $n\geq 2$, is called a Ricci soliton if there exists a vector field $X$ such that the following condition holds:
\begin{equation}\label{r7}
Ric+\frac{1}{2}\pounds_Xg=\lambda g,
\end{equation}
where $\lambda$ is a constant and $\pounds$ denotes the Lie derivative. The vector field $X$ is called potential vector field. The Ricci solitons are self-similar solutions to the Ricci flow, which is developed by Hamilton \cite{HA88,HA82}.  If $X=\nabla f$, for some function $f\in C^\infty(M)$, then (\ref{r7}) takes the form
\begin{equation}\label{g1}
\nabla f^2+Ric=\lambda g.
\end{equation}
It is called gradient Ricci soliton and the function $f$ is called potential function. If $f$ is constant then the gradient Ricci soliton is simply the Einstein manifold. An Einstein manifold with constant $f$ is called trivial Ricci soliton. The Ricci soliton is called shrinking, steady or expanding if $\lambda>0$, $\lambda=0$ or $\lambda<0$, respectively. If the gradient Ricci soliton is shrinking, then after rescaling  the metric $g$ we assume that $\lambda=\frac{1}{2}$. Then (\ref{g1}) takes the form
\begin{equation}\label{eq10}
\nabla f^2+Ric=\frac{1}{2} g.
\end{equation}
Shrinking Ricci soliton has great impact on the topology of the manifold. One of the first result in this connection is the Myers's classical theorem which says that a compact Riemannian manifold with positive Ricci curvature has finite fundamental group. Fern\'andez-L\'opez and Garc\'ia-R\'io \cite{FG08} proved, using universal cover, that the fundamental group is finite for the compact shrinking Ricci soliton. Zhang \cite{ZH07} also gave an alternative proof. The same result has been proved earlier by Lott \cite{LO03} for gradient Ricci soliton. Naber \cite{NA06} generalized this result and proved for bounded Ricci curvature. Wylie \cite{WY07} proved the strongest result that any complete shrinking Ricci soliton has finite fundamental group. Locally conformally flat Ricci soliton have been investigated immensely in the last few years. Cao et al. \cite{CWZ11} Petersan and Wylie \cite{PW10} and Zhang \cite{ZH09} classified locally conformally flat gradient Ricci soliton and they showed that it is isometric to $\mathbb{S}^n,\mathbb{R}^n,\mathbb{R}\times\mathbb{S}^{n-1}$ or one of their quotients.  For more results on Ricci soliton see \cite{SC20,ZH09}.\\
\indent In this paper, we have showed that a complete conformally flat gradient shrinking Ricci soliton with reciprocal of the potential functional being subharmonic and the manifold satisfying linear volume growth is isometric to the Euclidean sphere. As a result of this, we have derived that the conformally flat restriction condition can be neglected in $3$-dimension gradient shrinking Ricci soliton and in case of $4$-dimension, half-conformally flat condition is needed. We have also deduced a compactness criteria for gradient shrinking Ricci soliton with quadratic volume growth. \\
The set of all finitely integrable function on the manifold $M$ is denoted by $L^1(M)$, i.e.,
$$L^1(M)=\Big\{f|M\rightarrow\mathbb{R}:\int_Mf<+\infty \Big\}.$$
\begin{thm}\label{th2}
Let $(M,g,f)$ be an $n$-dimensional complete non-flat gradient shrinking Ricci soliton with scalar curvature $R$ satisfying $R\in L^1(M)$. If $\frac{1}{f}$ is subharmonic, then $M$ is compact Einstein manifold with positive Ricci curvature and scalar curvature $R=\frac{n}{2}$. Furthermore,  if $M$ is conformally flat, then $M$ is a space form. Furthermore, the manifold is isometric to the $n$-dimensional Euclidean sphere $\mathbb{S}^n$.
\end{thm}
Since, in $3$-dimension, the Weyl conformal curvature vanishes, the following result can be stated:
\begin{cor}
A $3$-dimensional gradient shrinking Ricci soliton $(M,g,f)$ with $R\in L^1(M)$ and $\Delta \frac{1}{f}\geq 0$  is simply connected. 
\end{cor}
A Riemannian manifold $M$ is said to have linear volume growth, if  
$$\limsup_{r\rightarrow\infty}\frac{Vol(B(p,r))}{r}<+\infty,$$
for every $p\in M$. Various properties of Riemannian manifold with linear volume growth have been studied by many authors, for example see \cite{SO98,SO00}. If we use the linear volume growth condition, then the curvature restricted condition of Theorem \ref{th2} can be omitted.
\begin{thm}\label{th4}
Suppose $(M,g,f)$ is an $n$-dimensional complete conformally flat gradient shrinking Ricci soliton with linear volume growth. If $\frac{1}{f}$ is subharmonic, then $M$  is isometric to $\mathbb{S}^n$. 
\end{thm}
The Weyl curvature tensor $W$ in $4$-dimensional oriented Riemannian manifold can be decomposed into two irreducible components, $W^+$ and $W^-$ under the action of special orthogonal group \cite{BE87}. If $W^+$ or $W^-$ vanishes, then the manifold is called half-conformally flat. Hitchin \cite{HI74} classified $4$-dimensional compact half-conformally flat Einstein manifolds. 
\begin{thm}\cite[Theorem 13.30]{BE87}\label{thm1}
Let $M$ be a compact half-conformally flat four dimensional Einstein manifold with positive scalar curvature. Then $M$ is isometric to $\mathbb{S}^4$ or $\mathbb{C}P^2$.
\end{thm}
Now by using the Theorem \ref{th2}, Theorem \ref{th4} and the Theorem \ref{thm1}, we conclude the following:
\begin{cor}
Suppose $(M,g,f)$ is a $4$-dimensional complete gradient shrinking Ricci soliton with $\frac{1}{f}$ is subharmonic. If $M$ is half-conformally flat and any one of the following conditions holds\\
(i)  $R\in L^1(M)$\\
(ii) $M$ has linear volume growth,\\
 then $M$ is isometric to $\mathbb{S}^4$ or $\mathbb{C}P^2$. Therefore, $M$ has trivial fundamental group.
\end{cor}
Recently, Zhang \cite{ZH20} proved the following result:
\begin{thm}\cite{ZH20}\label{thm2}
Suppose $(M,g)$ is a compact four-dimensional Einstein manifold with $Ric=g$ and the sectional curvature $K\leq \frac{1}{12}+\frac{\sqrt{30}}{8}$. Then $(M,g)$ is isometric to either $\mathbb{S}^4$ or $\mathbb{RP}^4$ or $\mathbb{CP}^2$.
\end{thm}
By using the Theorem \ref{thm2} and proof of Theorem \ref{th4}, the following can be stated:
\begin{cor}
Let $(M,g,f)$ be a $4$-dimensional complete gradient shrinking Ricci soliton with linear volume growth and sectional curvature $K\leq \frac{1}{12}+\frac{\sqrt{30}}{8}$. If $\frac{1}{f}$ is subharmonic, then $(M,g)$ is isometric to either $\mathbb{S}^4$ or $\mathbb{RP}^4$ or $\mathbb{CP}^2$.
\end{cor}
Next using the proof of the Theorem \ref{th2}, we will prove the following compactness criteria:
\begin{thm}\label{th3}
If an $n$-dimensional gradient shrinking Ricci soliton $(M,g,f)$ realizes the following conditions\\
(i)  $Vol(B_r(p))\leq Cr^2$ for all $p\in M$ and $r>0$, where $C>0$ is a constant,\\
(ii) $f$ satisfies the inequality $\Delta f^2\leq |\nabla f|^2$,\\
 then $M$ is compact.
\end{thm}
\section{Proof of the results}
To prove the Theorem \ref{th2}, we need the following result:
\begin{thm}\cite{GO69}\label{t1}
A conformally flat compact manifold with metric of positive Ricci
curvature and constant scalar curvature is a space form.
\end{thm}
\begin{proof}[\textbf{Proof of Theorem \ref{th2}}]
Taking trace of (\ref{eq10}), we get
\begin{equation}\label{eq1}
R+\Delta f=\frac{n}{2}.
\end{equation}
Again Hamilton \cite{HA95} proved that the potential function $f$ in a shrinking gradient Ricci soliton, after adding some constant if necessary, satisfies
\begin{equation}\label{eq2}
R+|\nabla f|^2=f.
\end{equation}
Since $M$ is shrinking Ricci soliton, the scalar curvature is non-negative \cite{ZH09} and also $f$ satisfies the asymptotic behavior \cite{CZ10}, i.e., there exist positive constants $C$ and $C'$ such that $f$ obeys the estimation
\begin{equation}\label{eq8}
\frac{1}{4}\Big[(d(x,p)-C)_+\Big]^2\leq f(x)\leq \frac{1}{4}\Big(d(x,p)+C'\Big)^2,
\end{equation}
where $C_+=max(C,0)$.
 Again since $M$ is non-flat \cite{ CH09}, $R>0$. Thus  (\ref{eq2}), yields $f>0$. Therefore, in normal coordinate we calculate 
$$\Big(\frac{1}{f} \Big)_i=-\frac{1}{f^2}f_i,\
\nabla \Big(\frac{1}{f}\Big)=-\frac{1}{f^2}\nabla f,$$
and
$$\Big(\frac{1}{f} \Big)_{ii}=\Big(-\frac{1}{f^2}f_i\Big)_i=\frac{2}{f^3}f_i^2-\frac{1}{f^2}f_{ii}.$$
Therefore
\begin{equation}\label{eq9}
\Delta \Big(\frac{1}{f}\Big)=\frac{2}{f^3}|\nabla f|^2-\frac{1}{f^2}\Delta f.
\end{equation}
Similarly, we calculate
$$\Big(\frac{1}{f^2} \Big)_i=-\frac{2}{f^3}f_i,\
\nabla \Big(\frac{1}{f^2}\Big)=-\frac{2}{f^3}\nabla f,$$
and
$$\Big(\frac{1}{f^2} \Big)_{ii}=\Big(-\frac{2}{f^3}f_i\Big)_i=\frac{6}{f^4}f_i^2-\frac{2}{f^3}f_{ii}.$$
Hence
\begin{equation}
\Delta \Big(\frac{1}{f^2}\Big)=\frac{6}{f^4}|\nabla f|^2-\frac{2}{f^3}\Delta f.
\end{equation}
For a fixed constant $c>0$, as introduced in \cite{CLY11}, we define the function $u$ in $M$ by
$$u(x)=R(x)-\frac{c}{f(x)}-\frac{nc}{f^2(x)}\ \text{ for }x\in M.$$
Hence, using weighted Laplacian, (\ref{eq1}) and (\ref{eq2}), we obtain
\begin{eqnarray}\label{eq4}
\nonumber\Delta_f(f^{-1})&=& \Delta f^{-1}-g(\nabla f,\nabla f^{-1})\\
\nonumber&=& \frac{2}{f^3}|\nabla f|^2-\frac{1}{f^2}\Delta f-g(\nabla f,-\frac{1}{f^2}\nabla f)\\
\nonumber&=& \frac{2}{f^3}|\nabla f|^2-\frac{1}{f^2}\Delta f+\frac{1}{f^2}|\nabla f|^2\\
\nonumber&=& \frac{2}{f^3}|\nabla f|^2-\frac{1}{f^2}(\frac{n}{2}-R)+\frac{1}{f^2}(f-R)\\
\nonumber&=& \frac{2}{f^3}|\nabla f|^2-\frac{n}{2f^2}+\frac{1}{f}\\
&=& \frac{1}{f}-\frac{1}{f^2}\Big(\frac{n}{2}-\frac{2|\nabla f|^2}{f} \Big)
\end{eqnarray}
and
\begin{eqnarray}\label{eq3}
\nonumber\Delta_f(f^{-2})&=& \Delta f^{-2}-g(\nabla f,\nabla f^{-2})\\
\nonumber&=& \frac{6}{f^4}|\nabla f|^2-\frac{2}{f^3}\Delta f+\frac{2}{f^3}|\nabla f|^2\\
\nonumber&=& \frac{6}{f^4}|\nabla f|^2-\frac{2}{f^3}(\frac{n}{2}-R)+\frac{2}{f^3}(f-R)\\
\nonumber&=& \frac{6}{f^4}|\nabla f|^2-\frac{n}{f^3}+\frac{2}{f^2}\\
&=& \frac{2}{f^2}-\frac{1}{f^3}\Big(n-\frac{6|\nabla f|^2}{f} \Big).
\end{eqnarray}
From \cite{ENM08}, we get
$$\Delta R=g(\nabla f,\nabla R)+R-2|Ric|^2.$$
Then, we conclude
\begin{equation}\label{eq5}
\Delta_f R=R-2|Ric|^2.
\end{equation}
Now, (\ref{eq4}) and (\ref{eq5}) together imply that for any constant $c>0$
\begin{eqnarray}\label{eq6}
\nonumber\Delta_f\Big(R-\frac{c}{f}\Big)&=& R-2|Ric|^2-\frac{c}{f}+\frac{c}{f^2}\Big(\frac{n}{2}-\frac{2|\nabla f|^2}{f} \Big)\\
&\leq & R-\frac{c}{f}+\frac{c}{f^2}\Big(\frac{n}{2}-\frac{2|\nabla f|^2}{f} \Big).
\end{eqnarray}
Then, from above inequality and (\ref{eq3}), we obtain
\begin{eqnarray}\label{eq7}
\nonumber\Delta_f u &=& \Delta_f R-\Delta_f\Big(\frac{c}{f}\Big)-\Delta_f\Big(\frac{nc}{f^2}\Big)\\
\nonumber&\leq & R-\frac{c}{f}+\frac{c}{f^2}\Big(\frac{n}{2}-\frac{2|\nabla f|^2}{f} \Big)-\Big(\frac{2}{f^2}-\frac{1}{f^3}\Big(n-\frac{6|\nabla f|^2}{f} \Big)\Big)\\
&=& u-\frac{cn}{f^3}\Big(\frac{f}{2}-n\Big)-\frac{c}{f^4}\Big(2f+6n\Big)|\nabla f|^2.
\end{eqnarray}
Now for sufficiently small $c>0$ we can consider $u$ as positive inside $B(p,C+3n)$. Let us suppose that there is a point $x_0\in M-B(p,C+3n)$ where $u$ attains its negative minimum. Then evaluation (\ref{eq7}) at the point $x_0$ and using maximum principle, we get $\frac{f(x_0)}{2}\leq n$. Again (\ref{eq8}) shows that $f(x_0)\geq\frac{9n^2}{4}$, which leads to a contradiction. Hence, we conclude that $u\geq0$ in $M$. Therefore for sufficiently small $c>0$, we obtain 
\begin{equation}\label{eq13}
R\geq \frac{c}{f}+\frac{cn}{f^2}\geq \frac{c}{f}.
\end{equation}
Now construct a cut-off function, introduced in \cite{CC96}, $\varphi_r\in C^\infty_0(B(p,2r))$ in $M$ with the property
\[ \begin{cases} 
	  0\leq \varphi_r\leq 1 &\text{ in }B(p,2r)\\
      \varphi_r=1  & \text{ in }B(p,r) \\
      |\nabla \varphi_r|\leq\frac{1}{2r}& \text{ in }B(p,2r) \\
      \varphi_r=0 &  \text{ in }\partial B(p,2r).
   \end{cases}
\]
Since $\frac{1}{f}$ is subharmonic, it follows that 
\begin{eqnarray*}
0&\leq & \int_{B_{2r}(p)}\varphi_r^2\frac{1}{f}\Delta\Big(\frac{1}{f}\Big)\\
&=&-\int_{B_{2r}(p)}\varphi_r^2|\nabla \frac{1}{f}|^2-2\int_{B_{2r}(p)}\varphi_r\frac{1}{f}g\Big(\nabla\varphi_r,\nabla\frac{1}{f}\Big).
\end{eqnarray*}
Then we have
\begin{eqnarray*}
&&\int_{B_{2r}(p)}\varphi_r^2|\nabla \frac{1}{f}|^2 \leq  -2\int_{B_{2r}(p)}\varphi_r\frac{1}{f}g\Big(\nabla\varphi_r,\nabla\frac{1}{f}\Big)\\
&&\leq 2\Big(\int_{B_{2r}(p)}\varphi_r^2|\nabla \frac{1}{f}|^2 \Big)^{1/2}\Big(\int_{B_{2r}(p)}\big(\frac{1}{f}\big)^2|\nabla \varphi_r|^2 \Big)^{1/2}.
\end{eqnarray*}
Therefore, we obtain
\begin{eqnarray}\label{in1}
\nonumber\int_{B_{r}(p)}|\nabla \frac{1}{f}|^2 &\leq & \int_{B_{2r}(p)}\varphi_r^2|\nabla \frac{1}{f}|^2 \leq 4\int_{B_{2r}(p)}\big(\frac{1}{f}\big)^2|\nabla \varphi_r|^2\\
&\leq & \frac{1}{r^2}\int_{B_{2r}(p)}\big(\frac{1}{f}\big)^2.
\end{eqnarray}
Now from (\ref{eq13}), we get
$$R\geq \frac{nc}{f^2}\text{ in }M.$$
Then the inequality (\ref{in1}) reduces to
\begin{equation}\label{eq11}
\int_{B_{r}(p)}|\nabla \frac{1}{f}|^2\leq \frac{1}{ncr^2}\int_{B_{2r}(p)}R.
\end{equation}
Since $R\in L^1(M)$, taking limit $r\rightarrow\infty$ in both sides we get
$|\nabla \frac{1}{f}|^2=0$, which implies that $f$ is constant. Therefore, we conclude, using Myers's theorem, from (\ref{eq10}) that $M$ is compact with constant scalar curvature $R=\frac{n}{2}$. Also from (\ref{eq10}) we state that $M$ is Einstein. Now for the second part, if $M$ is conformally flat, then we conclude from Theorem \ref{t1}, that $M$ is the space form. Again, the Ricci curvature of $M$ is positive, therefore, $M$ has constant positive sectional curvature. Thus, $M$ is isometric to $\mathbb{S}^n$.
\end{proof}
\begin{proof}[\textbf{Proof of Theorem \ref{th4}}]
The scalar curvature $R$ of gradient shrinking Ricci soliton satisfies the following inequality \cite{CZ10}:
$$\int_{B_r(p)}R\leq\frac{n}{2}Vol(B(p,r)),$$
for any $p\in M$. Therefore, (\ref{eq11}) implies that
\begin{equation}
\int_{B_{r}(p)}|\nabla \frac{1}{f}|^2\leq \frac{1}{2cr^2}Vol(B(p,2r)).
\end{equation}
On the other hand, the manifold $M$ satisfies the linear volume growth condition. Hence, taking the limit $r\rightarrow\infty$ of the above inequality, we get $|\nabla \frac{1}{f}|^2=0$. Therefore, using the same argument of the proof of Theorem \ref{th2}, we conclude that the result.
\end{proof}
For the proof of Theorem \ref{th3}, the following results are needed:
\begin{lem}\label{lm1}
If $f>0$ and $\Delta f^2\leq |\nabla f|^2$, then $(\frac{1}{f})$ is subharmonic.
\end{lem}
\begin{proof}
Given that $\Delta f^2\leq |\nabla f|^2.$
Then
\begin{eqnarray*}
0&\geq & \Delta f^2-6|\nabla f|^2\\
&=& (\Delta f^2-2|\nabla f|^2)-4|\nabla f|^2\\
&=& f\Delta f-2|\nabla f|^2.
\end{eqnarray*}
Therefore
$$ 2|\nabla f|^2-f\Delta f\geq 0.$$ 
Since $f>0$, the above inequality implies that
$$\frac{2}{f^3}|\nabla f|^2-\frac{1}{f^2}\Delta f\geq 0.$$
The right side of the above inequality is $\Delta(\frac{1}{f})$, which is non-negative. Therefore, $\frac{1}{f}$ is subharmonic.
\end{proof}
\begin{thm}\cite{KA82}\label{th1}
Let $f\in C^\infty(M)$ be a nonnegative subharmonic function. Then either $f$ is constant or 
$$\liminf_{r\rightarrow\infty}\frac{1}{r^2}\int_{B_r(p)}f^pdV=+\infty,$$
for every $p\in M$ and $p>1.$
\end{thm}

\begin{proof}[\textbf{Proof of Theorem \ref{th3}}]

Since $M$ is non-flat shrinking Ricci soliton, the scalar curvature is positive \cite{CH09,ZH09}. Thus (\ref{eq2}) implies that $f>0$. Then from (\ref{eq13}), we have
\begin{equation}
\frac{nc}{f^2}\leq R.
\end{equation}

Furthermore, Lemma \ref{lm1} implies $\frac{1}{f}$ is subharmonic. Again
$$\Delta\Big(\frac{1}{f^2}\Big)=2\frac{1}{f}\Delta\frac{1}{f}+2|\nabla \frac{1}{f}|^2.$$
Hence, $\frac{1}{f^2}$ is also subharmonic.
Now, the scalar curvature estimation in \cite{CZ10} implies that
\begin{equation}
\int_{B_r(p)}\frac{nc}{f^2}dV\leq \int_{B_r(p)} RdV\leq \frac{n}{2}Vol(B(p,r)).
\end{equation} 
Since $M$ has quadratic volume growth, i.e., $Vol(B(p,r))\leq C_1r^2$, where $C_1$ is a positive constant, the above inequality implies that
$$c\int_{B_r(p)}\frac{1}{f^2}dV\leq \frac{C_1}{2}r^2,$$
i.e.,
\begin{equation}
\frac{1}{r^2}\int_{B_r(p)}\frac{1}{f^2}dV\leq C',
\end{equation}
where $C'=C_1/2c$ is a positive constant.
This implies that
$$\liminf_{r\rightarrow\infty} \frac{1}{r^2}\int_{B_r(p)}\frac{1}{f^2}dV<+\infty,$$ and it contradicts Theorem \ref{th1}. Thus we conclude that $f$ must be constant. Hence (\ref{eq10}) implies $Ric= \frac{1}{2}g$. Therefore Myers's Theorem \cite{AM57} implies that the manifold $M$ is compact.
\end{proof}

\end{document}